\documentclass[final]{article}       
\usepackage[utf8]{inputenc}

\usepackage[T1]{fontenc}

\usepackage{amsmath}
\usepackage{dsfont}
\usepackage{epsfig}
\usepackage{pstricks}
\usepackage{amssymb}
\usepackage{graphicx,amsfonts,psfrag,fancyhdr,layout,appendix,subfigure}
\usepackage{chapterbib}
\usepackage{makeidx}
\usepackage{framed}
\usepackage{tikz}
\usepackage{pgfplots,pgfplotstable}
\usepackage{rotating}
\usepackage{algorithm}
\usepackage{algpseudocode}
\usetikzlibrary{arrows}

\fancypagestyle{plain}{
\fancyhf{}

\fancyfoot[LE,RO]{\thepage}
}
\makeatletter
\def\cleardoublepage{\clearpage\if@twoside \ifodd\c@page\else
    \hbox{}
    \thispagestyle{plain}
    \newpage
    \if@twocolumn\hbox{}\newpage\fi\fi\fi}
    
\newcommand{\customlabel}[2]{%
\protected@write \@auxout {}{\string \newlabel {#1}{{#2}{}}}}

\makeatother \clearpage{\pagestyle{plain}\cleardoublepage}

\makeindex

\newcommand{\RR}{{\rm I\hspace{-0.50ex}R}}

\newcommand{\NN}{{\rm I\hspace{-0.50ex}N}}
\newcommand{\Un}{{\rm I\hspace{-1.40ex}1}}
\newcommand{\integset}[2]{[\hspace{-0.35ex}[ #1,#2 ]\hspace{-0.35ex}]}
\newcommand{\epsi}{\varepsilon}

\newtheorem{defin}{Definition}
\newtheorem{theorem}{Theorem}
\newtheorem{lemma}{Lemma}[theorem]
\newtheorem{proposition}{Proposition}[theorem]

\newenvironment{proof}{{\bf Proof}}{$\Box$.}
\definecolor{MyGray}{rgb}{0.96,0.97,0.98}
   {\framed\textbf{\textsf{ TO BE REMEMBERED\vspace*{1em}}}\newline }%
   {\endframed}

\begin{document}

\title{Differentiation and regularity of semi-discrete optimal transport with respect to the parameters of the discrete measure}
\author{Fr\'ed\'eric de Gournay \and Jonas Kahn \and L\'eo Lebrat}


\maketitle

\abstract{ This paper aims at determining under which conditions the semi-discrete optimal transport is twice differentiable with respect to the parameters of the discrete measure and exhibits numerical applications. The discussion focuses on minimal conditions on the background measure to ensure differentiability. We provide numerical illustrations in stippling and blue noise problems.
}

\section{Introduction}

Optimal transport \cite{monge, kantorovich} is a blossoming subject that has known major breakthroughs these last decades. Its applications range from finance \cite{pages2012optimal}, mesh generation \cite{du2002grid}, PDE analysis \cite{jordan1998variational} and imaging \cite{rubner2000earth,wang2013linear} to machine learning and clustering \cite{solomon2014wasserstein,flamary2016wasserstein}.

This paper is limited to the semi-discrete case, which consists in transporting discrete measures (Dirac masses) towards a background measure. However we allow more general background measures than the densities with respect to the Lebesgue measure that are usually found in the literature. In this setting, we prove second order differentiability of the optimal transport distance for an arbitrary cost with respect to the locations of the Dirac masses.

Precursors include De Goes \cite{de2014geometric} (Proposition 2.5.4) who has given  the formula of the Hessian in the Euclidean setting. However he has given no proof of existence. We will make use of the framework developed by  Kitagawa and Mérigot \cite{kitagawa2016newton} to overcome these restrictions.

As a by-product, we obtain the second order differentiability conditions for the so-called energy of a Voronoi diagram. 
The latter remark generalizes results presented in \cite{du2006convergence,liu2009centroidal} to higher dimensions and lower regularity of the background measure.

\subsection{Semi-discrete optimal transport}

The optimal transport \cite{villani} between two probability measures $\mu$ and $\nu$ defined respectively on the spaces $X$ and $Y$ with cost $c:Y\times X\rightarrow \RR^+$ is the minimization problem :
\begin{align}
    \label{eq::OT}
    \min_{\gamma \in \Pi(\nu,\mu)} \int_{Y\times X} c(y,x)d\gamma(y,x),
\end{align}
where $\Pi(\nu,\mu)$ is the set of positive measures on $Y\times X$ with marginal distributions on $Y$ (resp. $X$) equal to $\nu$ (resp. $\mu$), that is :
\[\gamma \in \Pi(\nu,\mu) \Leftrightarrow \begin{cases}
    \int_Y\phi(y)d\mu(y)=\int_{Y\times X}\phi(y)d\gamma(y,x) \quad \forall \phi \in L^1(\mu)\\ \text{ and } \\ 
\int_X\psi(x)d\nu(x)=\int_{Y\times X}\psi(x)d\gamma(y,x)\quad \forall \psi \in L^1(\nu)
\end{cases}.
\]

Intuitively, a coupling $\gamma$ may be seen as a way to transport the mass of $\mu$ to the mass of $\nu$. Specifically $\gamma(B,A)$ is the mass moved from $A$ to $B$. 

Interpreting $\phi$ and $\psi$ as Lagrange multipliers of the constraint  $\Pi(\mu,\nu)$ 
and using a standard inf-sup inversion (see \cite{villani} for details), one derives the Kantorovitch dual problem:
\[\sup_{\phi(y)+\psi(x)\le c(y,x)}\int_Y\phi(y)d\mu(y) +\int_X\psi(x)d\nu(x). \]
When $\phi$ is given, it can be explicitly solved in $\psi$. The corresponding optimal $\psi^\star$ is
\[\psi^\star(x)=\inf_{y\in Y} (c(y,x)-\phi(y)).\]
Hence the problem can be cast as
\begin{equation}
\label{eq::defin:g}
\sup_{\phi} g(\phi) \quad \text{ with } \quad g(\phi)=\int_Y\phi(y)d\mu(y) +\int_X \inf_{y\in Y}  (c(y,x)-\phi(y))d\nu(x),
\end{equation}
where the function $g$, as a dual function, is naturally concave. 

Suppose now that the support of $\nu$ is included in a bounded convex Lipschitz set $\Omega \subset \RR^d$, and that $\mu$ is a discrete measure on $\RR^d$, that is, given $n \in \NN$, there exists $z=(z^i)_{i=1..n}$ with $z^i\in \RR^d$ and $m=(m^i)_{i=1..n} \in \RR^n$ such that
\[\mu =\sum_{i=1}^n m^i\delta_{z^i}, \]
where $\delta_{z^i}$ is a Dirac measure located at $z^i$. In this case the set of test functions $\phi$ can be identified to $\RR^n$, so that $\phi=(\phi^i)_{i=1..n} \in \RR^n$. The computation of $\psi^\star$ is then easily given as
\begin{equation}
\label{eq::defin::psi}
\psi^\star(x)=\min_{i\in \integset{1}{n}}\psi^i(x) \quad \text{ with }\psi^i(x)=c(z^i,x)-\phi^i
\end{equation}

Finally introducing the Laguerre tessellation \cite{aurenhammer1987power} defined by its cells
\begin{equation}
\label{eq::Laguerre}
\mathcal L_i(z,\phi)=\{x\in \Omega \text { such that } \psi^i(x)\le \psi^j(x)\quad \forall j\in \integset{1}{n}\},
\end{equation}
we have $\psi^\star=\psi^i$ on $\mathcal L_i(z,\phi)$ so that the final formulation of the optimal transport problem \eqref{eq::OT} in the semi-discrete setting is
\begin{align}
    & \phantom{=} \sup_\phi g(\phi,z,m)  \qquad \text{with} \notag \\
    \label{eq::defin:g2} g(\phi,z,m) & =\sum_{i=1}^n \int_{\mathcal L_i(z,\phi)} \!\! \left(c(z^i,x)-\phi^i\right)\frac{1}{\#{\mathcal M}^{-1}(\{x\})}d\nu(x) + \sum_{i=1}^n \phi^i m^i,
\end{align}
where $\#{\mathcal M}^{-1}(\{x\})$ is defined in Section~\ref{sec::hyp} as counting factor of the number of Laguerre cells that intersect at point $x$.

The Laguerre cells $\mathcal{L}_i $ associated to an optimal $\phi$ in the maximization \eqref{eq::defin:g} are the ``arrival'' zones of the mass at each $z^i$ by the corresponding coupling $\gamma$. We call such a tessellation an optimal Laguerre tessellation.

We aim at studying the differentiation properties up to the second order of $g(\phi,z,m)$ with respect to its parameters. The differentiation with respect to $m$ is rather straightforward and will not be discussed hereafter. The second order differentiability of $g$ with respect to $\phi$ is known \cite{merigot2013comparison,de2012blue,levy2015numerical} and proved in \cite{kitagawa2016newton}. This proof mainly uses that the Laguerre cells $\mathcal L_i$ are the intersection for all $j\ne i$ of the sub-level sets (with respect to the value of $\phi^i-\phi^j$) of the function $x\mapsto c(z^i,x)-c(z^j,x)$. Using the co-area formula, the authors are able to compute the differential of $g$ with respect to $\phi$. Differentiating with respect to $z$ is more involved and is the main goal of the present paper.

\subsection{Link with Voronoi diagrams}
\label{intro::voronoi}

The Voronoi diagram $\{\mathcal{V}_i \}_i$ is the special case of the Laguerre tessellation when $\phi=0$ and the cost is the square Euclidean distance, as can be seen from 
definitions \eqref{eq::defin::psi} and \eqref{eq::Laguerre}.

Moreover, Aurenhammer \cite{Aurenhammer1998} has proved that for any choice of Lagrange multipliers $\phi$, there is a vector of masses $(m^i)_i$ such that the solution of \eqref{eq::defin:g} is given by $\phi$. Indeed, the choice $m^i = \nu(\mathcal L_i(z, \phi))$ turns $\phi$ into a critical point of the concave function $g$.

Hence the Voronoi diagram is the optimal Laguerre tessellation for the choice of mass $\tilde m^i := \nu(\mathcal{V}_i(z))$.

Moreover, the mass  $\tilde m$ is optimal in the following sense:
\begin{align}
    \label{eq::energy}
    g(z, \tilde m, 0) & = \sup_{\phi} g(z, \tilde m, \phi) \notag   \\
                      & = \inf_{m} \sup_{\phi} g(z, m, \phi).
\end{align}
A first explanation is that $\phi$ is a Lagrange multiplier for the mass constraint, so that the solution for $\phi = 0$ will be optimal for the mass. 
Another more physical interpretation is that, without mass constraints, the best way to transport a measure $\nu$ to a finite number of points is to send each part of $\nu$ to its closest neighbour. Hence we build the Voronoi diagram of the points.

The expression \eqref{eq::energy} has been coined as the energy of the  Voronoi diagram:
\begin{equation}
\label{eq::defin::Voronoienergy}
G_S(z)=g(0,z,\tilde m)=\sum_i \int_{\mathcal V_i(z)}\Vert x-z^i\Vert^2d\nu(x)
\end{equation}

Finding critical points of this energy $G$  is also known as the centroidal Voronoi tessellation (CVT) problem. Indeed, at a critical point $\bar z$ of $G$, each $\bar z^i$ is the barycenter of $\mathcal V_i(\bar z)$ with respect to the measure $\nu$:
\[\bar z^i \int_{\mathcal V_i(\bar z)}d\nu(x) = \int_{\mathcal V_i(\bar z)}xd\nu(x).\]

Results of second order differentiability of $G$ with respect to $z$ has been proven in \cite{liu2009centroidal} and inferred in many different previous papers \cite{iri1984fast,asami1991note,du1999centroidal}. 
However those papers do not tackle the question of the regularity of $\nu$. Moreover the Voronoi setting is Euclidean. On both of these points, our work is a generalization, since differentiability of $g$ immediately implies differentiability of $G$.

\subsection{Organization of the paper}
In Section~\ref{sec::Main results Section}, the main result is given. The hypotheses needed to ensure second order differentiability are given in \ref{sec::hyp}, the result is stated in Section~\ref{sec::MainResult}, Theorem~\ref{prop::2differentiability} and is reformulated in the Euclidean case in Section~\ref{sec::euclideancase}. The rest of Section~\ref{sec::Main results Section} is devoted to the proof of Theorem~\ref{prop::2differentiability}. Section~\ref{sec::numresult} presents some numerical results.

\section{Second order differentiability}
\label{sec::Main results Section}
The main goal of this section is to state, in Theorem~\ref{prop::2differentiability}, the sufficient conditions that ensure differentiability of second order of the function 
\begin{align*}
    \tilde g(\phi,z)  & = \int_{\Omega} \psi^{\star}(x) \mathrm{d}\nu(x) \\
                     & = \int_{\Omega } \min_i c(z^i, x) - \phi^i \ \mathrm{d}\nu(x),
\end{align*}
which yields immediately the second order derivatives of $g=\tilde g+\phi\cdot m$ defined in \eqref{eq::defin:g}.

\subsection{Hypotheses and notation}
\label{sec::hyp}


In order to state the hypotheses required for Theorem~\ref{prop::2differentiability}, additional notation is required.
\begin{figure}
\begin{tikzpicture}[scale=5]

\coordinate (A124) at (0,0);
\coordinate (A132) at (0.2,0.8);
\coordinate (A134) at (2,0.2);
\fill [color=gray!10] (A124) -- (A132) -- (A134) --cycle;
\draw (A124) -- (A132) node[midway,above,sloped] {$\partial \mathcal L_1 \cap \partial \mathcal L_2$} ;
\draw (A132) -- (A134) node[midway,above,sloped] {$e^{1,3}$};
\draw (A134) -- (A124) node[midway,below,sloped] {$e^{1,4}$};
\draw [dashed] (0,0.9) -- ++(2,0.1) node[midway,below,sloped] {$e^{1,10}$};
\draw [dashed] (A124) -- ++(-0.05,-0.2) node[midway,above,sloped]  {$e^{1,2}$} ;
\draw [dashed] (A124) -- ++(-0.05,-0.005)  ;
\draw [dashed] (A132) -- ++(0.05,0.2) ;
\draw [dashed] (A132) -- ++(-0.18,0.06)  ;
\draw [dashed] (A134) -- ++(0.18,-0.06) ;
\draw [dashed] (A134) -- ++(0.18,0.018)  ;

\draw (0.6,0.3) node {$\mathcal L_1$};
\draw (0.1,0.7) node {$\mathcal L_2$};
\draw (1.6,0.8) node {$\mathcal L_3$};
\draw (1.5,0) node {$\mathcal L_4$};

\coordinate (varepsBL) at (-.12,-.2);
\coordinate (varepsTL) at (-.16,.17);
\coordinate (varepsBR) at (2.18,.03);
\coordinate (varepsTR) at (2.15,.4);
\fill [color=gray!90,fill opacity=0.2] (varepsBL) -- (varepsBR) -- (varepsTR) -- (varepsTL) ;
\draw (varepsBL) -- (varepsBR);
\draw [color=gray!50] (varepsBL) -- (varepsBR);
\draw [color=gray!50] (varepsTL) -- (varepsTR);
\draw (2.07,.05) node[above] {$\mathcal N_{1,4}(\varepsilon)$};
\draw[<->,color=gray!90] (2,.01) -- (1.96,.38);
\end{tikzpicture}
    \caption{Typical example of $e^{ik}$ and the neighborhoods $\mathcal N_{ik}(\varepsilon)$ for $i=1$.}
\label{fig::explain-notations}
\end{figure}
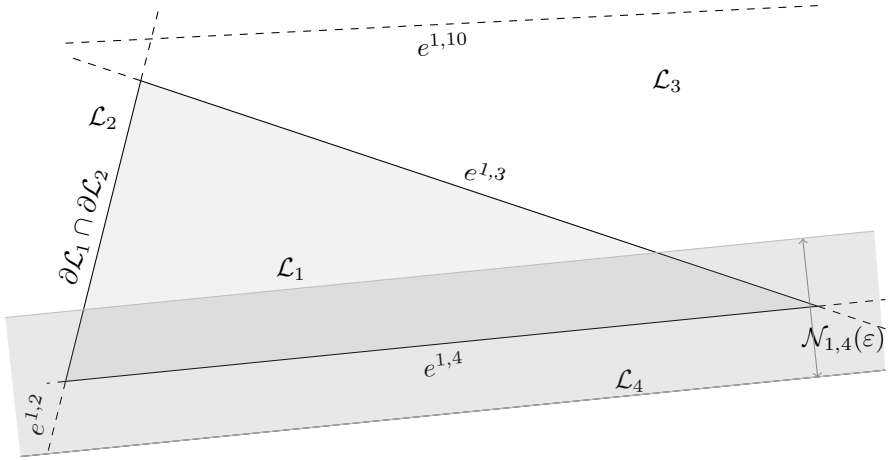

For all $x\in \Omega$, denote $\mathcal M(x)$ the subset of $\integset{1}{n}$  given by
\begin{equation}
\label{eq::defin:M}
\mathcal M(x)=\left \{ i \in \integset{1}{n} \text{ s.t } c(z^i,x)-\phi^i\le c(z^j,x)-\phi^j \quad \forall {j} \in \integset{1}{n} \right\}.
\end{equation}
The Laguerre cell $\mathcal L_i(z,\phi)$ is then exactly given by
 \[x \in \mathcal L_i(z,\phi) \Leftrightarrow i \in \mathcal M(x).\]

For $i\ne k \in \integset{1}{n}$, denote
\begin{equation}
\label{eq::defin:eik}
e^{ik}:=\{x  \text{ s.t. } c(z^i,x)-\phi^i=c(z^k,x)-\phi^k\}.
\end{equation}

Note that $\mathcal L_i \cap \mathcal L_k$ is included in $e^{ik}$ but the converse fails to be true. First notice that $e^{ik}$ is not included in $\Omega$ whereas the Laguerre cells are included in $\Omega$ by definition. Second  $e^{ik}$ is only the ``competition zone" between the $i^{th}$ and the $k^{th}$ Laguerre cells but it may (and will) happen that $x\in e^{ik}$ is included in $\mathcal L_j$ for some other $j$ and in neither $\mathcal L_i$ nor $\mathcal L_k$.

For all $i$ and $k$, we denote the $\varepsilon$-neighborhood of $e^{ik}$ by $\mathcal{N}_{ik}(\varepsilon)$. By convention, $e^{i0}$ is the boundary of $\Omega$ and thus $\mathcal{N}_{i0}(\varepsilon)$ is an $\varepsilon$-neighborhood of $\partial \Omega $. 
Figure~\ref{fig::explain-notations} illustrates these geometric objects.

The geometric hypotheses on the Laguerre tessellation that are required for second order differentiability are :
\begin{defin}[\ref{2nd-order-diff}]
\customlabel{2nd-order-diff}{Diff-2}
We say that hypothesis $\eqref{2nd-order-diff}$ holds iff
\begin{itemize}
\item 
for all $1\le i\le n$, $(z,x)\mapsto c(z^i,x)$ is $W^{2,\infty}( B(z_0,r)\times\Omega)$, where $B(z_0,r)$ is a ball around the point $z_0$.
\item there exits $\epsi>0$ such that for all $1\le k\ne i\le n$, $ \forall x\in e^{ik}$ 
\begin{equation}
\Vert \nabla_x c(z^i,x)-\nabla_x c(z^k,x)\Vert \ge \epsi .
\label{2nd-order-diffa}
\tag{\ref{2nd-order-diff}-a}
\end{equation}
\item for all $i$, there exists $s$ small enough and $C$ such that for all $0\le k\ne j\le n$, for all $0<\varepsilon,\varepsilon'<s$, it holds
\begin{equation}
    \left \vert \mathcal N_{ik}(\varepsilon) \cap \mathcal N_{ij}(\varepsilon')\right \vert \le C \varepsilon \varepsilon' .
\label{2nd-order-diffb}
\tag{\ref{2nd-order-diff}-b}
\end{equation}
and
\begin{equation}
    \lim_{\varepsilon\rightarrow 0}\sigma\left(  e^{ik} \cap \mathcal N_{ij}(\varepsilon)\right) =0 ,
\label{2nd-order-contb}
\tag{\ref{2nd-order-diff}-c}
\end{equation}
where $\sigma$ is the $d-1$ Hausdorff measure.
\end{itemize}
\end{defin}

The geometric hypothesis for second order continuity is
\customlabel{2nd-order-cont}{Cont-2}
\begin{defin}[\ref{2nd-order-cont}]
We say that hypothesis $\eqref{2nd-order-cont}$ holds iff there exists $C>0$ such that for all $i,j$ if $\sigma$ is the $d-1$ Hausdorff measure, then
\begin{equation}
\sigma(e^{ij}\cap \Omega) \le C.
\label{2nd-order-conta}
\tag{\ref{2nd-order-cont}}
\end{equation}
\end{defin}

 \subsection{Main result}
 \label{sec::MainResult}
Directional derivative of $\tilde g$ can  be obtained using very mild assumptions on the cost function $c$ and the approximated measure $\nu$.

\begin{proposition}
\label{prop::differentiability}
Set $\Omega$ a bounded Lipschitz convex set. 
Suppose that for all $i$ and for $\nu$-almost every $x$, the function $z\mapsto c(z^i,x)$ is differentiable around $z_0$ and that there exists $h \in L^1(\Omega,\nu)$ with $|\nabla_zc(z^i,x)|\le h(x)$ $\nu$-a.e. for all $z$ around $z_0$. Then $\tilde g$ is directionally derivable with derivative given by :
\begin{eqnarray*}
&&\lim_{t\rightarrow 0^+}\frac{\tilde g((\phi+td_\phi,z_0+td_z))-\tilde g(\phi,z_0)}{t}\\
&&=\sum_{A \subset \integset{1}{n} } \int\displaylimits_{\mathcal M^{-1}(A)} \min_{i\in A}\left( \langle \nabla_{z}c(z_0^i,x),d_z^i
\rangle-d_\phi^i \right)d\nu(x)
\end{eqnarray*}

    If $\nu(\mathcal M^{-1}(A))=0$ for each $A$ of cardinal $\ge 2$, it holds that $\tilde g$ is differentiable and
\[\nabla_z \tilde g(z_0,\phi)=\sum_{i=1}^{n} \int_{\mathcal L_i(z_0,\phi)} \nabla_{z}c(z_0^i,x) d\nu(x)\quad \text{ and } \quad \frac{\partial \tilde g}{\partial \phi_i}(z_0,\phi)= - \nu(\mathcal L_i(z_0,\phi)) \]
If in addition for $\nu$-almost every $x$, the function $z\mapsto \nabla_{z}c(z^i,x)$  is continuous around $z_0$ for all $i$, then  $\tilde g$ is $C^1$ around $(\phi,z_0)$.
\end{proposition}

The theorem ensuring second order differentiability is: 
\begin{theorem}
\label{prop::2differentiability}
If the hypotheses of Proposition~\ref{prop::differentiability} and $\eqref{2nd-order-diff}$ hold, and if $\nu$ admits a density with respect to the Lebesgue measure which is $W^{1,1}(\Omega)\cap L^\infty(\Omega)$ then $\tilde g$ is twice differentiable and the formula for the Hessian is given by :
\begin{eqnarray*}
\frac{\partial^2 \tilde g}{\partial a\partial b}&=& \!\!\!\!\!\!\! \sum_{i\ne j \in \integset{1}{n}} \int_{\partial {\mathcal L}_i\cap \partial {\mathcal L}_j} \!\!\!\!\!\!\!\! \frac{(\partial_a c(z^j,x)-\partial_a\phi^j-\partial_a c(z^i,x)+\partial_a \phi^i)(\partial_b c(z^i,x)-\partial_b\phi^i)}{\Vert \nabla_x c(z^i,x)-\nabla_x c(z^j,x)\Vert}m(x)d\sigma \\
&+&\sum_i \int_{{\mathcal L}_i} \frac{\partial^2 c}{\partial a\partial b}(z^i,x)d\nu(x)
\end{eqnarray*} 
where $\sigma$ is the $d-1$ Hausdorff measure, $m$ is the density of $\nu$ and $a$ or $b$ have to be replaced by $\phi$ or $z$. If in addition $\eqref{2nd-order-cont}$ holds and if the density $m$ is $C^{0}(\Omega)$ then $\tilde g$ is $C^{2}$.
\end{theorem}

\subsection{The Euclidean case}

\label{sec::euclideancase}
This section deals specially with the Euclidean case $c(z^i,x) = \frac{1}{2}\| z^i - x\|_2^2$. Then we have
\begin{proposition}\label{eq:derivatives}
If $c(y,x)=\frac{1}{2}\Vert y-x\Vert_2^2$, all the hypotheses are verified if $\nu$ admits a $C^{0}(\Omega)\cap W^{1,1}(\Omega)$ density with respect to the Lebesgue measure, and $z^i\ne z^j$ for $i\ne j$ and there is no Laguerre cell of zero Lebesgue volume. In this case the first order formulas are given by  :
\[\frac{\partial \tilde g}{\partial z_i}=\int_{\mathcal L_i}\left(z^i-x\right) d\nu(x)  \quad \text{ and } \quad \frac{\partial \tilde g}{\partial \phi_i}= - \nu(\mathcal L_i) \]
and the second order formula are given by :
\begin{eqnarray*}
 \frac{\partial^2 \tilde g}{\partial \phi^i\partial \phi^j}&=& \int_{\partial {\mathcal L}_i\cap \partial {\mathcal L}_j} 
\frac{m(x)}{\Vert z^i-z^j\Vert}d\sigma \quad \text{ if } i \ne j\\
 \frac{\partial^2 \tilde g}{(\partial \phi^i)^2}&=& -\sum_{j \ne i} \frac{\partial^2 \tilde g}{\partial \phi^i\partial \phi^j} \\
 \frac{\partial^2 \tilde g}{\partial \phi^i\partial z^j}&=& - \int_{\partial {\mathcal L}_i\cap \partial {\mathcal L}_j} 
\frac{(z^j-x)m(x)}{\Vert z^i-z^j\Vert}d\sigma \quad \text{ if } i \ne j\\
 \frac{\partial^2 \tilde g}{\partial \phi^i\partial z^i}&=& -\sum_{j \ne i} \frac{\partial^2 \tilde g}{\partial \phi^j\partial z^i}\\ 
 \frac{\partial^2 \tilde g}{\partial z^i\partial z^j}&=& \int_{\partial {\mathcal L}_i\cap \partial {\mathcal L}_j} 
\frac{(z^j-x)(z^i-x)^Tm(x)}{\Vert z^i-z^j\Vert}d\sigma \quad \text{ if } i \ne j\\
 \frac{\partial^2 \tilde g}{(\partial z^i)^2}&=& \nu(\mathcal L_i)-\sum_{j \ne i} \int_{\partial {\mathcal L}_i\cap \partial {\mathcal L}_j} 
\frac{(z^i-x)(z^i-x)^Tm(x)}{\Vert z^i-z^j\Vert}d\sigma \\
\end{eqnarray*}

\end{proposition}
\begin{proof}

In the Euclidean case the regularity assumption on $c$ is trivially satisfied. Moreover $e^{ik}$ is an hyperplane and $\Omega$ is bounded so that~\eqref{2nd-order-conta} is trivial.

A direct computation shows that
\[\Vert \nabla_x c(z^i,x)-\nabla_x c(z^j,x)\Vert  = \Vert z^i-z^j\Vert,\]
which is non zero by hypothesis and hence uniformly non-zero, so that Hypothesis \eqref{2nd-order-diffa} is satisfied.

For $j,k \ne 0$, the sets $\mathcal N_{ik}(\varepsilon)$ are $\varepsilon$-neighbourhoods of the hyperplane $e^{ik}$, so that~\eqref{2nd-order-diffb} and~\eqref{2nd-order-contb} are verified as soon as the hyperplanes $e^{ik}$ and $e^{ij}$ are different for $j\ne k$. On the other hand, it is impossible that $e^{ik} = e^{ij}$ for any $ j\ne k$. Indeed, by the definition~\eqref{eq::defin:eik}, it would mean that $z^i$,$z^k$ and $z^l$ are aligned and that the Laguerre cell corresponding to the point between the two others has empty interior, contradicting the hypotheses of the theorem.

Similarly, if $e^{ik}\cap \partial \Omega$ is not reduced to at most two points, by the convexity of $\Omega$, the set $\Omega$ lies on one side of $e^{ik}$ and one of the two Laguerre cells $\mathcal L_i$ or $\mathcal L_k$ is therefore empty. This final argument proves the case~\eqref{2nd-order-diffb} and ~\eqref{2nd-order-contb} for $j=0$. 

Now let $A$ be of cardinal $\ge 2$. Let $i$ and $k$ belong to $A$, then $\mathcal M^{-1}(A)$ is included in $e^{ik}$, but $e^{ik}$ is an hyperplane which is of zero Lebesgue measure, hence $\nu(\mathcal M^{-1}(A)) = 0$.

The rest of the hypotheses of Proposition~\ref{prop::differentiability} is trivial. 
\end{proof}

As proved in \cite{kitagawa2016newton}, the constant $C$ appearing in \eqref{2nd-order-diffb} depends on the minimal angle between the intersection of two competition zones $e^{ik}$ and $e^{il}$. This constant is non-zero since there is a finite number of such intersections and it drives the $C^{2,\alpha}$ regularity of the function $g$.
  
\subsection{Technical lemmata}
This section is devoted to proving two technical lemmata, the first one ensures second-order differentiability of the function $\tilde g$ and the second one ensures continuity. 
In this section, fix $i$, fix a $C^\infty$ mapping $t\mapsto (z(t),\phi(t))$ that we aim at deriving at time $t=0$ . Set $s$ small enough and consider only $t\in [0,s]$. Throughout this section the objects that depend on $(z,\phi)$ (say of the Laguerre cell $\mathcal L_j(z,\phi)$) will be written as depending on $t$ (with the obvious notation $\mathcal L_j(t)$). Denote
\[u^{ik}_t(x):=c(z^i(t),x)-\phi^i(t)-\left(c(z^k(t),x)-\phi^k(t)\right).\]
Denote $(u^{ik}_t)^{-1}(0):=\{x\text{ s.t. } u^{ik}_t(x)=0\}$. Note that $(u^{ik}_0)^{-1}(0)=e^{ik}$, where $e^{ik}$ is defined in~\eqref{eq::defin:eik}, Section~\ref{sec::hyp}.
\begin{lemma}
\label{lemma:deriv-integral}
Suppose that the Laguerre tessellation verifies~\eqref{2nd-order-diff}. Let 
\[\xi:t\mapsto \int_{\mathcal L_i(t)}f(x,t)dx,\] with $f$ in $W^{1,1} \cap L^\infty$ then $\xi$ is derivable at time $t=0$ with :
\[\partial_t \xi(0)=\sum_k \int_{{\mathcal L}_i(0)\cap {\mathcal L}_k(0)} \frac{\partial_t u^{ik}_0(x)}{\Vert \partial_x u^{ik}_0(x)\Vert}f(x,0)d\sigma+\int_{\mathcal L_i(0)}\partial_tf(x,0)dx,\]
where $\sigma$ is the $d-1$ Hausdorff measure on ${\mathcal L}_i(0)\cap {\mathcal L}_k(0)$.
\end{lemma}
\begin{lemma}
\label{lemma:continuity-integral}
Suppose the Laguerre tesselation verifies~\eqref{2nd-order-cont}. Let  $f$ be continuous, then $\xi:t\mapsto \int_{\mathcal L_i(t)\cap\mathcal L_k(t)}f(x,t)d\sigma$ is continuous. 
\end{lemma}
These lemmata are proven using tools of differential geometry via a diffeomorphism $\theta$ that maps approximatively $\mathcal L_i(0)$ to $\mathcal L_i(t)$. The organization of this section is as follows: In Section~\ref{subsec:construct-flow} the diffeormorphism is built, and it is shown that $\theta(\mathcal L_i(0))\simeq \mathcal L_i(t)$. The lemmata are then proven in Section~\ref{sec::proof:lemma}.

\subsubsection{Construction of the flow}
\label{subsec:construct-flow}

 For any $k \neq i, k> 0$, \eqref{2nd-order-diffa} ensures that $\Vert \partial_x u^{ik}_0 (x) \Vert$ is uniformly non-zero on $(u^{ik}_0)^{-1}(0)$. By Lipschitz regularity of $\partial_x u$, $\Vert \partial_x u^{ik}_t (x) \Vert$ is uniformly non-zero for all  $x \in \mathcal N_{ik}(s), t \in [0,s]$, provided $s$ is chosen small enough. Hence the vector field defined as :
\[
V^k_t(x):=-\partial_tu^{ik}_t(x) \frac{\partial_x u^{ik}_t(x)}{\Vert \partial_x u^{ik}_t(x) \Vert^2} \quad  \forall x \in \mathcal N_{ik}(s), t \in [0,s],
\]
is Lipschitz and can be extended as wanted outside $\mathcal{N}_{ik}(s)$.
%
The flow $\theta^k$ associated to $V^k_t$ is defined as :
\begin{equation}
\label{eq::defin::flow}
\theta^k_0(x)=x \text{ and }\partial_t \theta^k_t(x)=V_t^k(\theta^k_t(x)).
\end{equation}

The flow $\theta^k$ preserves the level-sets of $u^{ik}_t$ in the sense that for all $x$, the mapping $t\mapsto u^{ik}_t(\theta^k_t(x))$ is a constant as long as $\theta^k_t(x)$ remains in $\mathcal N_{ik}(s)$. Hence the flow $\theta^k_t$ preserves the competition zone between $\mathcal L_i(t)$ and $\mathcal L_k(t)$. 

The objective is to build a flow $\theta$ which preserves the whole boundary of $\mathcal L_i(t)$. To that end, introduce:
\[
\mathcal N_{ik}^\star(\varepsilon) =\bigcup_{0\le j \ne k \le n} \mathcal N_{ij}(\varepsilon),
\]
and denote $\delta_k(x)$ the distance function to $\mathcal{N}_{ik}^\star(0)$
\[
\delta_k(x)=\inf_y\left\{\Vert y-x\Vert, \text{s.t. } y\in \partial \Omega \bigcup\left( \bigcup_{1\le j\ne k\le n}e^{ij} \right)\right\}.
\]

Set $\zeta$ a non-decreasing $C^\infty(\RR^+,\RR)$ function equal to zero  on $[0,1/2]$ and to one on $[1,+\infty[$ and for all $x,0\le t \le s$ define
\[
\tilde V^k_t(x)=\zeta(\frac{\delta_k(x)}{s})V^k_t(x) \quad V_t(x)= \sum_{k=1}^n \tilde V^k_t(x)
\]

Then $V_t$ is equal to $V^k_t$ on $\mathcal N_{ik}(s)\setminus \mathcal N_{ik}^\star(s)$. One can safely interpret that $V_t=V^k_t$ on the edge of $\mathcal L_i(t)$ that is shared with $\mathcal L_k(t)$ and has been smoothed to zero on every corner of $\mathcal L_i(0)$.
\begin{figure}\label{fig::CutOff}
\centering

 \caption{Example of the vector fields $V^3_t$ and $\tilde V^3_t$ for $i = 1$}
\end{figure}

\medskip
Denote $\theta$ the flow associated to $V$. We claim that
\vspace{-0.2cm}
\begin{proposition}
\label{prop:main}
There exists $C,c>0$ such that for all $s$ small enough, for all $k$, and $0\le t\le cs$, the symmetric difference between $\theta_t(\mathcal L_i(0))$ and $\mathcal L_i(t)$ has Lebesgue volume bounded by $Cst$.
\end{proposition}
{\bf Proof}
Note first that the vector field $V_t$ is always zero on $\partial \Omega$ so that $\theta_t(\Omega)=\Omega$ for all $t$. In the sequel $C_v$ denotes an upper bound of the velocity of $\theta$ and $\theta^k$. Set $c\le 1/C_v$, then for all $k$ :
\begin{equation}
\label{eq::app}
\theta^k_t(e^{ik}) \subset \mathcal{N}_{ik}(s), 
\end{equation}
and then $\theta^k_t(e^{ik}) = (u^{ik}_t)^{-1}(0)$.
Let $x\in \theta_t(\mathcal L_i(0))\Delta \mathcal L_i(t)$ and denote $x_0$ such that $x=\theta_t(x_0)$. We claim that there exists $k\in \integset{1}{n}$ and $0\le r_1,r_2\le t$ such that $u^{ik}_{r_1}(\theta_{r_1}(x_0))=0$ and $u^{ik}_{r_2}(\theta_{r_2}(x_0)) \ne 0$.

Indeed, if for instance $x\in \theta_t(\mathcal L_i(0))$ but $x\notin \mathcal L_i(t)$, then trivially $x_0 \in \mathcal L_i(0)$, meaning that for all $k$, $u^{ik}_0(x_0)\le 0$. But $x\notin \mathcal L_i(t)$ means that there exists a $k$ such that $u^{ik}_t(\theta_t(x_0))>0$. The continuity of the mapping $t\mapsto u^{ik}_t(\theta_t(x))$ ensures that for some ${r_1}$ we have $u^{ik}_{r_1}(\theta_{r_1}(x_0))=0$. The other case is done the same way.

Since $\theta$ has bounded velocity, 
\[
\forall r\le t \quad \Vert \theta_{r}(x_0)-\theta_{r_1}(x_0)\Vert \le C_v|r-r_1|.
\]

By \eqref{eq::app}, since $\theta_{r_1}(x_0) \in (u^{ik}_{r_1})^{-1}(0)$, then $\theta_{r_1}(x_0) \in \mathcal N_{ik} (C_vr_1)$, then 
\begin{equation}
\label{equa::smallneighbourhood}
\forall 0\le r\le t,\quad \theta_{r}(x_0)\in \mathcal N_{ik}(2C_vt).
\end{equation}

Upon reducing $c$ by a factor $2$, $\theta_{r}(x_0)\in \mathcal N_{ik}(s)$. We now claim that 
\begin{equation}
\label{eq::app2}
\text{there exists } 0\le r_3\le t\text{ such that }\theta_{r_3}(x_0)\in \mathcal N^\star_{ik}(s).
\end{equation}
 Indeed, if it is not the case, then for all $r$ $V_r(\theta_r(x_0))=V^k_r(\theta_r(x_0))$ and then $\theta_r(x_0)=\theta^k_r(x_0)$ and hence
 $u^{ik}_{r}(\theta_{r}(x_0))$ is a constant which is in contradiction with $u^{ik}_{r_2}(\theta_{r_2}(x_0)) \ne 0$ and $u^{ik}_{r_1}(\theta_{r_1}(x_0)) = 0$. 
 Using the bounded velocity of $\theta$, and \eqref{eq::app2} we conclude that 
 $x=\theta_t(x_0)\in \mathcal N^\star_{ik}(s+C_vt)\subset \mathcal N^\star_{ik}(2s).$ 
Finally, using  \eqref{equa::smallneighbourhood}, we obtain 
\[
x\in \mathcal N^\star_{ik}(2s)\cap \mathcal N_{ik}(C_vt) = \bigcup_{j\ne k} \left(\mathcal N_{ij}(2s)\cap \mathcal N_{ik}(C_vt)\right).
\]
By hypothesis \eqref{2nd-order-diffb}, the last set has volume bounded by $Cst$ for some constant depending on $C_v$, the maximum velocity of $\theta^k$ and $\theta$. Since $C_v$ may be chosen independently of $s$ when $s$ is small enough, then $C$ is independent of $s$ and $t$.
$\Box$.
\subsubsection{Proof of lemmata}
\label{sec::proof:lemma}
We are now ready to tackle the proof of Lemmata~\ref{lemma:deriv-integral} and ~\ref{lemma:continuity-integral} in this section.

{\bf Proof of Lemma~\ref{lemma:deriv-integral}}
In this proof, the rate of convergence of $\frac{o(t)}{t}$ towards $0$ depends on $s$ (as $s^{-1}$).
Let $f$ in $L^\infty(\Omega\times \mathbb{R})$ with gradient in $L^1(\Omega\times \mathbb{R})$ and $s$ small enough. For all $t\le cs$, Proposition~\ref{prop:main} asserts
\begin{eqnarray*}
\int_{{\mathcal L}_i(t)} f(x,t)dx&=& \int_{\theta_t({\mathcal L}_i(0))} f(x,t)dx + \mathcal O(st)\\
&=&\int_{{\mathcal L}_i(0)} f(\theta_t(x),t)|det(J\theta_t(x))|dx+ \mathcal O(st)
\end{eqnarray*}

Where $J\theta_t$ is the Jacobian matrix of $\theta_t$. 

Using $\theta_t(x)=x+tV_0(x)+o_{L^\infty}(t)$, we then have (see \cite{henrot2006variation})
\[
f(\theta_t(x))=f(x,0)+t\partial_tf(x,0)+t\nabla_x f(x,0)\cdot V_0(x)+o_{L^1}(t),
\]
\[
|det(J\theta_t(x))|=1+tdiv(V_0)+o_{L^\infty}(t), 
\]
where $o_{L^a}(t)$ is a time dependent function that, when divided by $t$ goes towards zero in $L^a$ norm as $t$ goes to zero.
The rate of convergence depends on the Lipschitz norm of $V_t$ which scales as $s^{-1}$.

Then finally
\begin{eqnarray}
&&\int_{{\mathcal L}_i(t)} f(x,t)dx- \int_{{\mathcal L}_i(0)} f(x,0) dx 
=t \alpha_f(s) +o(t)+\mathcal O(st) 
\label{eq:degeu:Jonas}
\\ &&\text{ with }\alpha_f(s) =\int_{{\mathcal L}_i(0)}  \left(\partial_tf(x,0)+\nabla_x f(x,0)\cdot V_0(x)+f(x,0)div(V_0)\right)dx.
\nonumber
\end{eqnarray}
Recall that $V_0$ depends on $s$, hence $\alpha_f$ depends on $s$.

A Stokes formula yields
\[\alpha_f(s) =\int_{{\mathcal L}_i(0)}  \partial_tf(x,0)dx +\int_{\partial \mathcal L_i(0)} f(x,0)V_0(x)\cdot n_i d\sigma \]
This formula is true for Lipschitz domain and $\mathcal L_i(0)$ is Lipschitz because each $e^{ik}$ is Lipschitz as can be proven by a an implicit function theorem using~\eqref{2nd-order-diffb}.

Denote $Y(s)=\{ x \in \partial \mathcal L_i(0)\cap \partial \mathcal L_k(0) \text{ s.t. } \zeta(\frac{\delta_k(x)}{s})\ne 1\}$. Since
\[Y(s)\subset \bigcup_{j\ne k}\left(\mathcal N_{ij}(s) \cap e^{ik}\right),\]
we know $\displaystyle \lim_{s \rightarrow 0} \sigma(Y(s))=0$ by \eqref{2nd-order-contb}.

Since $f$ is in $W^{1,1}$, its trace on $e^{ik}$ is in $L^1(e^{ik})$ for the measure $\sigma$ \cite{brezis}. Hence, noticing that
\[1_{\partial L_i(0)}V_0(x)=\sum_{k\ne i} 1_{\partial \mathcal L_i(0)\cap \partial \mathcal L_k(0)} \zeta(\frac{\delta_k(x)}{s})V^k_0(x),
\]
the dominated convergence theorem asserts that $\alpha(s)$ converges as $s$ goes to zero towards
\[\lim_{s \rightarrow 0} \alpha_f(s) :=\alpha_f(0)=\int_{{\mathcal L}_i(0)}  \partial_tf(x,0)dx +\sum_{k\ne i}\int_{\partial \mathcal L_i(0)\cap \partial \mathcal L_k(0)} f(x,0)V^k_0(x)\cdot n_i d\sigma. \] 
Now denote,
\[r(t)=t^{-1}\left(\int_{{\mathcal L}_i(t)} f(x,t)dx- \int_{{\mathcal L}_i(0)} f(x,0) dx\right).\]
Let $t$ go to zero in \eqref{eq:degeu:Jonas}, we have $\limsup_{0^+} r(t)=\alpha_f(s)+\mathcal O(s)$ and $\liminf_{0^+} r(t)=\alpha_f(s)+\mathcal O(s)$. Letting $s$ goes to zero shows that $\lim_{0^+} r(t) $ exists and is equal to $\alpha_f(0) $ which proves lemma~\ref{lemma:deriv-integral}.

{\bf Proof of lemma~\ref{lemma:continuity-integral} }
The proof of this proposition owes so much to \cite{kitagawa2016newton}, proposition 3.2 that we even take the same notations. Consider the following partition of $ \mathcal L_i(t) \cap  \mathcal L_k(t)$ :
\begin{eqnarray*}
A_t&=&\left\{ \theta^k_t(x)\text{ s.t. }  \theta^k_r(x) \in \mathcal L_i(r) \cap  \mathcal L_k(r) \quad \forall r \in [0,s]\right\}\\
B_t&=&\left\{x\in  \mathcal L_i(t) \cap \mathcal L_k(t) \text{ s.t. } x\notin A_t\right\}
\end{eqnarray*}
It is clear that for all $t\le s$, we have
\[\xi(t)=\int_{A_t} f(x,t)d\sigma+\int_{B_t} f(x,t)d\sigma \]
In \cite{kitagawa2016newton}, in the first part of the proof of Proposition 3.2, the authors show that 
\[\lim_{s\rightarrow 0^+}\int_{A_s} f(x,s)d\sigma = \int_{A_0} f(x,0)d\sigma,\]
while actually controlling the convergence rate by the modulus of continuity of $f$. The reason is that $\theta^k_t$ is a Lipschitz diffeomorphism between 
$A_t$ and $A_0$ and that a change of variable allows to prove continuity. Note that no regularity assumption is made on the set $A_t$ except that its $d-1$ Hausdorff measure is bounded, which is exactly hypothesis~\eqref{2nd-order-conta}.
In order to prove that the sets $B_t$ are small with respect to the measure $d\sigma$, we follow a slightly simpler and quicker path than \cite{kitagawa2016newton} due to the fact that we use a stronger hypothesis in~\eqref{2nd-order-contb}.

First if $x=\theta^k_t(x_0)\in B_t$, then there exists $r\in [0,s]$,  such that :
\begin{equation}
\label{eq::BTcorner}
\theta^k_r(x) \in \partial \Omega \bigcup_{j\notin\{i,k\} } (u^{ij}_r)^{-1}(0)
\end{equation}

Indeed if $x\in \mathcal B_t$ then $\theta^k_r(x)$ is in $\mathcal L_i(r)\cap  \mathcal L_k(r)$ for $r=t$ and strictly outiside this set for some $r=r_1$. Recalling that

 \[ \mathcal L_i(t)\cap  \mathcal L_k(t)=\left\{ x \in \Omega \text{ s.t. } u^{ik}_t(x)=0 \text{ and } u^{ij}_t(x) \le 0 \text{ and } u^{kj}_t(x) \le 0 \quad \forall j\notin{\{i,k\}}\right \},\]
and that $\theta^k_t$ preserves the level-set $0$ of $u^{ik}_t$, then for some $r$, we must have by the intermediate value theorem either $\theta^k_r(x) \in \partial \Omega$ or $u^{ij}_r(\theta^k_r(x))=0$ or $u^{kj}_r(\theta^k_r(x))=0$. Finally if $u^{kj}_r(\theta^k_r(x))=0$, then  implies that $u^{ij}_r(\theta^k_r(x))=0$, since $u^{ik}_r(\theta^k_r(x))=0$.

 Suppose that we are in the case $u^{ij}_r(\theta^k_r(x_0))=0$ in~\eqref{eq::BTcorner}, then $ 
\theta^k_r(x_0)$ is in  $(u^{ij}_r)^{-1}(0)$, which, by finite velocity of $\theta^j$, is at distance at most $C_vs$ of $e^{ij}=(u^{ij}_0)^{-1}(0)$. By finite velocity of $\theta^k$, $x_0$ is at distance at most $C_vs$ of $\theta^k_r(x_0)$, meaning that $x_0$ is at distance at most $2C_vs$ of $e^{ij}$. Since $u^{ik}_0(x_0)=0$, we have that $x_0\in C_0:=e^{ik}\cap \mathcal N^\star_{ik}(2C_vs)$ which $d\sigma$ goes to zero as $s$ goes to zero by~\eqref{2nd-order-contb}. Finally 
\[B_t\subset \theta^k_t(C_0),\]
since $\theta^k_t$ is a Lipschitz diffeomorphism. Hence as $s$ goes to zero, $d\sigma (B_{t})$ goes to $0$ and hence.
\[\lim_{s\rightarrow 0^+}\int_{B_{t}} f(x,t)d\sigma =0= \lim_{s\rightarrow 0^+}\int_{B_0} f(x,0)d\sigma,\]
$\Box$
\subsection{Proof of the results of Section~\ref{sec::Main results Section}}
The goal of this section is to prove the different results of Section~\ref{sec::Main results Section}. We begin by Proposition~\ref{prop::differentiability}.
Suppose first $z^i\mapsto c(z^i,x)$ is differentiable $\nu$ a.e. for all $i$ and that $\nu$ is a positive Borelian measure of finite mass and rewrite $\tilde g$ as 
\[\tilde g(t)=\int_{\Omega} \psi^\star(z,\phi,x)d\nu(x)\]
where
 \[\psi^\star(z,\phi,x)=\min_i \psi^i(z,\phi,x) \quad \text{ with } \psi^i(z,\phi,x)=c(z^i,x)-\phi^i,\]
As the minimum of a finite number of differentiable functions, $\psi^\star$ is measurable and is $\nu$-a.e  directionally derivable with formula
\[{\psi^\star}'(x):=\lim_{t\rightarrow 0^+}\frac{\psi^\star(z+td_z,\phi+td_\phi,x)-\psi^\star(z,\phi,x) }{t } =\min_{i \in \mathcal M(x)} \langle \nabla \psi^i(x),d\rangle ,\]
with $d=(d_z,d_\phi)$ and $\nabla$ is the gradient with respect to $z$ and $\phi$. Recall for that purpose that $\mathcal M(x)$ is exactly the argmin of $\psi^i(x)$.
The function $(\psi^\star)'$ is seen to be measurable when rewritten as :
\[{\psi^\star}'(x)= \sum_{A \subset \integset{1}{n}} \Un_{\{ \mathcal M^{-1}(A)\}}(x) \min_{i\in A} \langle \nabla \psi^i(x),d\rangle,\]
the  set $\mathcal M^{-1}(A)$ being measurable since 
\[(\mathcal M(x)=A) \Leftrightarrow \left(\psi_i(x)=\psi^\star(x) \quad \forall i \in A \text{ and } \psi_i(x)>\psi^\star(x) \quad \forall i \in A^c\right)\]
A standard dominated convergence theorem asserts that the  directional derivative of $\tilde g$ exists and is given by : 
\[\tilde g'=\int_\Omega (\psi^\star)'(x)d\nu(x)\]
and we retrieve
\[\tilde g'=\sum_{A \subset \integset{1}{n}}  \int_{\mathcal M^{-1}(A)} \min_{i\in A} \langle \nabla \psi^i(x),d\rangle d\nu(x)\]
which is exactly the formula of Proposition~\ref{prop::differentiability}. When one supposes that $\nu(\mathcal M^{-1}(A))=0$ as soon as the cardinal of $A$ is strictly greater than $1$, then $\tilde g'$ is linear w.r.t $d$ and hence differentiable. In this case, we have $\mathcal  M^{-1}(\{i\})=\mathcal L_i(z,\phi)$ up to a set of zero $\nu$-measure, and hence
\[\nabla \tilde g(z,\phi)=\sum_{i=1}^n \int_{\mathcal L_i(z,\phi)} \nabla \psi_i(z,\phi,x)d\nu(x)\]
In order to prove the continuity of the gradient of $\tilde g$, we use the following technical lemma with $f=\nabla \psi_i(z,\phi,x)$.

\begin{lemma}
\label{lemma::continuity}
Suppose that $\nu(\mathcal M^{-1}(A))=0$ if $\#(A)\ge 2$. If $f$ is continuous with respect to $z,\phi$ for almost every $x$ and if there exists  $l\in L^1(\Omega,\nu)$ such that $|f(z,\phi,x)|\le l(x)$ $\nu$-a.e. for all $(z,\phi)$  then 
\[F_i : (z,\phi)\mapsto \int_{\mathcal L_i(z,\phi)}f(z,\phi,x)d\nu\] is continuous. 
\end{lemma}
{\bf Proof of Lemma~\ref{lemma::continuity}} First recall that
\[{\mathcal L}_i(z,\phi)=\{x \in \Omega \text{ s.t. } \psi^i(z,\phi,x) \le \psi^\star(z,\phi,x)\}.\]
For a sequence  $(z_n,\phi_n)$ that goes to $(z,\phi)$, denote
\[h_n(x)=f(z_n,\phi_n,x)\Un_{\psi^i(z_n,\phi_n,x)\le \psi^\star(z_n,\phi_n,x)}(x)\Un_{\Omega}(x),\]
and $h=\Un_{\mathcal L_i} f$.
Then 
\begin{align*}
    F_i(z_n, \phi_n)& = \int_{\Omega} h_n d\nu, & F_i(z, \phi)& = \int_{\Omega} h d\nu. 
\end{align*}
Moreover $h_n \leq l$ for all $n$.

If $x$ is such that $i \notin \mathcal M(x)$, that is $\psi^i(z,\phi,x) > \psi^\star(z,\phi,x)$, then $h_n(x)$ converges to $h(x)$.
If $x$ is such that $\mathcal M(x)=\{i\}$, then by continuity of $\psi^j(z,\phi,x)$ for all $j\ne i$, $\psi^i(z_n,\phi_n,x) = \psi^\star(z_n,\phi_n,x)$ for $n$ sufficiently large, hence $h_n(x)$ converges to $h(x)$. Then $h_n$ converges to $h=\Un_{\mathcal L_i} f$ except possibly on the sets where $\mathcal M^{-1}(A)$ is of cardinal greater or equal than $2$, which is, by hypothesis, of zero $\nu$-measure. 

Since $h_n \leq l$ for all $n$, a dominated convergence theorem ensures the continuity of the integral with respect to $(z,\phi)$.
$\Box$

The proof of Theorem~\ref{prop::2differentiability} is straightforward. We apply Lemma~\ref{lemma:deriv-integral} 
to $f(z,x)=\nabla_zc(z^i,x)m(x)$ or $f(z,x)=m(x)$, where $m$ is the density of $\nu$. Then we apply Lemma~\ref{lemma:continuity-integral} to the formula of the second order derivative in order to prove second order continuity.

\section{Numerical experiments}
\label{sec::numresult}

In this section we test a second order algorithm for the $2$-Wasserstein distance, when $c$ is the Euclidean cost.
Two problems will be solved: \textit{Blue Noise} and \textit{Stippling}. In both cases, we optimize a measure $\mu$ of the form $\mu(z,m)=\sum_{i=1}^n m^i \delta_{z^i}$, so that the 2-Wasserstein distance $W_2(\mu(z,m),\nu)$ is minimal.
\begin{itemize}
\item \textit{Blue Noise: } 
Here the weights $m^i$ are fixed. Hence the functional to minimize reads as :
\begin{equation}
\label{eq::bluenoise}
\inf_{z \in \RR^{nd}}G_B(z) \text{ with } G_B(z)=W_2(\mu(z,m),\nu)=\max_\phi g(\phi,z,m).
\end{equation}
\item \textit{Stippling: }
This problem consists in optimizing in $m$ and in $z$ simultaneously : 
\begin{equation*}
 \inf_{z \in \RR^{nd}} \inf_{m \in \Delta_n} W_2(\mu(z,m),\nu) ,
\end{equation*}
where $\Delta_n$ is the canonical simplex.
\end{itemize}

The \textit{Stippling} problem is actually easier than the \textit{Blue Noise} problem. Following the discussion of Section~\ref{intro::voronoi}, optimizing the mass amounts to set $\phi= 0$ and $\tilde m_i = \nu(\mathcal{V}_i)$:
\begin{equation}
\label{eq::stippling}
\inf_{z \in \RR^{nd}}G_S(z) \text{ with } G_S(z)=W_2(\mu(z,\tilde m),\nu)= g(0,z,\tilde m),
\end{equation}
where $G_S$ is the Voronoi energy defined in \eqref{eq::defin::Voronoienergy}. 
Hence no optimization procedure is required in $\phi$ and $\tilde m$ is merely given by computing the $\nu$-mass of each Voronoi cells.

\paragraph{Formulas}
Recall that in the Euclidean case, the formulas for $g$ boil down to
\begin{eqnarray*}
\frac{\partial g}{\partial z^i}&=&M^i(z^i-\bar z^i) \\
\frac{\partial^2 g}{\partial z^i z^j}&=&\int_{\partial {\mathcal L}_i\cap \partial {\mathcal L}_j} 
\frac{(z^j-x)(z^i-x)^Tm(x)}{\Vert z^i-z^j\Vert}d\sigma \text{ if }i\ne j\\
\frac{\partial^2 g}{\partial (z^i)^2}&=&M^i - \sum_{i\ne j}\frac{\partial^2 g}{\partial z^i z^j}
\end{eqnarray*}
where $M^i=\int_{\mathcal L_i} d\nu$ is the mass of the $i^{th}$ Laguerre cell and $\bar z^i=\int_{\mathcal L_i} x d\nu/M^i$ is its barycenter. 
\subsection{Lloyd's algorithm}
\begin{algorithm}
\caption{Lloyd's algorithm with Wolfe stabilization. \label{alg:mainalgo}}
\begin{algorithmic}[1]
\State \textbf{Inputs:} 
\State Initial guess $z^0$ 
\State target measure $\nu$
\State \textbf{Outputs:} 
\State An approximation  of the solution of \eqref{eq::bluenoise}.
\While{Until convergence}
\State Compute $\phi^\star(z^k)$.
\State Compute $\bar z^k_i$ the barycenter of the $i^{th}$ Laguerre cell $\mathcal L_i(z^k,\phi^\star(z^k))$.
\State Set $d^k=\bar z^{k}_i-z^{k}_i$
\State Set $\sigma^k=1$ and $z^{k+1}_i=z^{k}_i +\sigma^kd^k$
\While{ not \ref{eq::Wolfe} conditions fulfilled}
\State $\sigma^k=\sigma^k/2$ and $z^{k+1}_i=z^{k}_i +\sigma^kd^k$.
\EndWhile
\State $k=k+1$
\EndWhile
\State Return $z^{k}$.
\end{algorithmic}
\end{algorithm}

The gradient algorithm for computing the \textit{Blue Noise} (resp. the \textit{Stippling}  problem) is to move each point in the direction of the barycenter of its Laguerre cell (resp. Voronoi cell). Taking the diagonal metric given by the mass of the cells $M=(M^i)_{i=1..n}$ (which is a decent approximation of the Hessian), yields the following formula for the gradient of $G$
\begin{equation}
\label{eq::gradient with metric}
\langle a,b \rangle_M=\sum_i M^ia^ib^i \Longrightarrow \nabla G(z_k) =z^i-\bar z^i.
\end{equation}

A fixed step gradient Algorithm with step $1$ is to set each point $z^i$ exactly at the location of the barycenter $\bar z^i$. This algorithm is well known as a Lloyd-like or a relaxation algorithm  \cite{lloyd1982least,du1999centroidal}. An improvement of Lloyd's algorithm is to ensure a Wolfe step condition \cite{bertsekas}. 
\begin{equation}
\label{eq::Wolfe}
\tag{Wolfe}
G(z_{k+1})<G(z_k)+10^{-4}\langle \nabla G(z_k),z_{k+1}-z_k\rangle_M
\end{equation}
This naturally leads to algorithm~\ref{alg:mainalgo}. The only difference between the \textit{Stippling} and \textit{Blue Noise} problems lies in the choice of $\phi^\star(z)$. It is chosen equal to $0$ in the \textit{Stippling} problem and to $\textrm{argmax}_\phi g$ in the \textit{Blue Noise} problem.

Numerical experiment shows that it is not necessary to check for Wolfe's second condition which ensures that the step is not too small, indeed Lloyd's algorithm (and Newton's algorithm) have a natural step $\sigma^k=1$.
\subsection{Newton's Algorithm}
The second algorithm is a Newton algorithm. Denoting by $H$ the Hessian, in the \textit{Stippling} case, we have:
\[H_{zz}G_{S}=H_{zz}g\]
The computation of $H_{zz} G_B$ for the \textit{Blue Noise} case is more involved. A chain rule yields
\[H_{zz}G_{B}=H_{zz}g+H_{z\phi}g \cdot \nabla_z \phi^\star\]
The existence of $\nabla_z \phi^\star$ is given by an implicit function theorem, from
\[\nabla_\phi g(\phi^\star(z),z,m)=0.\]
Differentiating the above equation with respect to $z$ and applying the chain rule, we get
\[H_{z\phi} g+H_{\phi\phi} g\nabla_{z}\phi^\star=0\]
and hence
\[H_{zz}G_{B}=H_{zz}g-H_{z\phi} g  (H_{\phi\phi}g)^{-1}  H_{\phi z}g. \]
The implicit function theorem that proves existence of $\nabla_{z}\phi^\star$ requires the matrix $H_{\phi\phi}g$ to be invertible. Note that constant $\phi$ are always part of the kernel of $g$ but upon supposing that $\phi$ has zero average, the invertibility of $H_{\phi\phi}g$ is verified throughout the optimization procedure.

Once the Hessian is computed, the Newton algorithm with preconditioning by the matrix $M$ amounts to changing in Algorithm~\ref{alg:mainalgo} the descent direction $d^k$ by a solution to the linear problem 
 \begin{equation}
 \label{eq::system Newton}
 AM^{1/2}d^k=-M^{1/2}\nabla G(z^k) \quad \text{ with } A=(M^{-1/2}H_{zz}G M^{-1/2}),
 \end{equation}
and $\nabla G(z^k)$ is defined in~\eqref{eq::gradient with metric} as the gradient with respect to the metric $M$.
Newton's algorithm fails if the Hessian is not positive definite, hence we propose a work-around based on the conjugate gradient method on the system \eqref{eq::system Newton}. Recall that conjugate gradient method solve exactly the problem in the Krylov space and that the residues of the conjugate gradient method form an orthogonal basis of this Krylov space, hence are equal (up to a normalization procedure) to the Lanczos basis. Denote $\pi_n$ the projection on the Krylov space at iteration $n$, the matrix $\pi_n A \pi_n$ is tridiagonal in the Lanczos basis hence the computation of its determinant is a trivial recurrence \cite{saad2003iterative}. By monitoring the sign of the determinant throughout iterations one checks the positiveness of the matrix. The conjugate gradient algorithm is stopped whenever the matrix $A$ stops being positive definite. The descent direction is then given by
\[(\pi_n A\pi_n )M^{1/2}d^k=-\pi_n M^{1/2}\nabla G(z^k),\]
By convention for $n=0$, we solve $d^k=-\nabla G(z^k)$. If $A$ is positive, then the problem~\eqref{eq::system Newton} is solved exactly.
\subsection{Other considerations}
The computation of $\phi^{\star}$ in the \textit{Blue Noise} problem is a standard unconstrained concave maximization procedure with knowledge of second order derivatives. In order to compute $\phi^\star$ in a robust manner, we settled on a Levenberg-Marquardt type algorithm: denoting $H(\sigma)=H_{\phi\phi}g-\frac{1}{\sigma} Id$, we take as descent direction $-H(\sigma)^{-1}\nabla g(\phi)$, where $\sigma$ is reduced until Wolfe's first order conditions are met. 
In the \textit{Stippling} problem, the computation of $\phi^{\star}=0$ is trivial.

The Laguerre tessellation is computed by CGAL~\cite{cgal}. All the tests where performed using a standard Lena image as background measure $\nu$ which has been discretized as bilinear by pixel ($Q1$ finite element method). In the \textit{Blue Noise} problem, the mass $m$ is constrained to be equal to $\frac{1}{n}$ for all Diracs.

\subsection{Numerical results}
\subsection{Direct comparaison of the algorithms}

For the first example, we search the optimal positions of the Dirac masses for either the  \textit{Blue Noise} or \textit{Stippling} problem. Three methods are benchmarked, the Gradient method (Lloyd-like method), the Newton method discussed in the previous section and a \textit{LBFGS} method with the memory of the 8 previous iterations. Tests are performed for 1K and 10K points.
The evolution of the cost functions and the $L_2$ norms of the gradient are displayed throughout iterations. 
Figure~\ref{Fig:Voronoi-optim} displays the results obtained for the \textit{Stippling} problem whereas Figure~\ref{fig::OT-optim} displays the results for the \textit{Blue Noise} problem.

\pgfplotstableread[col sep=comma]{LL-1K-normgradient.dat}{\LLuKnormGradient}
\pgfplotstableread[col sep=comma]{LL-1K-costfunction.dat}{\LLuCostFunction}
\pgfplotstableread[col sep=comma]{LL-10K-normgradient.dat}{\LLdKnormGradient}
\pgfplotstableread[col sep=comma]{LL-10K-costfunction.dat}{\LLdCostFunction}

\begin{figure}[!hbt]
\begin{tabular}{ccc}
&$L_2$ norm of the gradient & Objective function \\
\begin{sideways}
\hspace{1cm} $1000$ points 
\end{sideways}
&
\begin{tikzpicture}[every node/.append style={font=\tiny}]
\begin{axis}[ymode=log,width=0.45\textwidth]
\addplot [color=black]
          table[x expr=\thisrow{iteration},y={Lloyd}]{\LLuKnormGradient};
\addplot [color=gray,dash pattern=on 1pt off 1pt,very thick]
          table[x expr=\thisrow{iteration},y={Newton}]{\LLuKnormGradient};
\addplot [color=lightgray,dash pattern=on 1pt off 1pt on 3pt off 3pt,very thick]
          table[x expr=\thisrow{iteration},y={LBFGS}]{\LLuKnormGradient};
\end{axis}
\end{tikzpicture}
&
\begin{tikzpicture}[every node/.append style={font=\tiny}]
\begin{axis}[width=0.45\textwidth,ymax=0.0002,xmax=200]
\addplot [color=black]
          table[x expr=\thisrow{iteration},y={Lloyd}]{\LLuCostFunction};
\addplot [color=gray,dash pattern=on 1pt off 1pt,very thick]
          table[x expr=\thisrow{iteration},y={Newton}]{\LLuCostFunction};
\addplot [color=lightgray,dash pattern=on 1pt off 1pt on 3pt off 3pt,very thick]
          table[x expr=\thisrow{iteration},y={LBFGS}]{\LLuCostFunction};
\end{axis}
\end{tikzpicture}
 \\
\begin{sideways}
\hspace{2cm} $10000$ points 
\end{sideways}
&
\begin{tikzpicture}[every node/.append style={font=\tiny}]
\begin{axis}[ymode=log,width=0.45\textwidth,legend style={ at={(0.5,-0.15)}, anchor=north, legend columns=2}]
\addplot [color=black]
          table[x expr=\thisrow{iteration},y={Lloyd}]{\LLdKnormGradient};
\addplot [color=gray,dash pattern=on 1pt off 1pt,very thick]
          table[x expr=\thisrow{iteration},y={Newton}]{\LLdKnormGradient};
\addplot [color=lightgray,dash pattern=on 1pt off 1pt on 3pt off 3pt,very thick]
          table[x expr=\thisrow{iteration},y={LBFGS}]{\LLdKnormGradient};
\legend{Lloyd,Newton,$8$-LBFGS}
\end{axis}
\end{tikzpicture}
&
\begin{tikzpicture}[every node/.append style={font=\tiny}]
\begin{axis}[width=0.45\textwidth,legend style={ at={(0.5,-0.15)}, anchor=north, legend columns=2},ymax=0.00002,xmax=200]
\addplot [color=black]
          table[x expr=\thisrow{iteration},y={Lloyd}]{\LLdCostFunction};
\addplot [color=gray,dash pattern=on 1pt off 1pt,very thick]
          table[x expr=\thisrow{iteration},y={Newton}]{\LLdCostFunction};
\addplot [color=lightgray,dash pattern=on 1pt off 1pt on 3pt off 3pt,very thick]
          table[x expr=\thisrow{iteration},y={LBFGS}]{\LLdCostFunction};
\legend{Lloyd,Newton,$8$-LBFGS}
\end{axis}
\end{tikzpicture}
\end{tabular}
\caption{\textit{Stippling} problem with 1Kpts (Top) and 10 Kpts (Bottom). Left : norm of the gradient, Right : evolution of the cost function.}
\label{Fig:Voronoi-optim}
\end{figure}
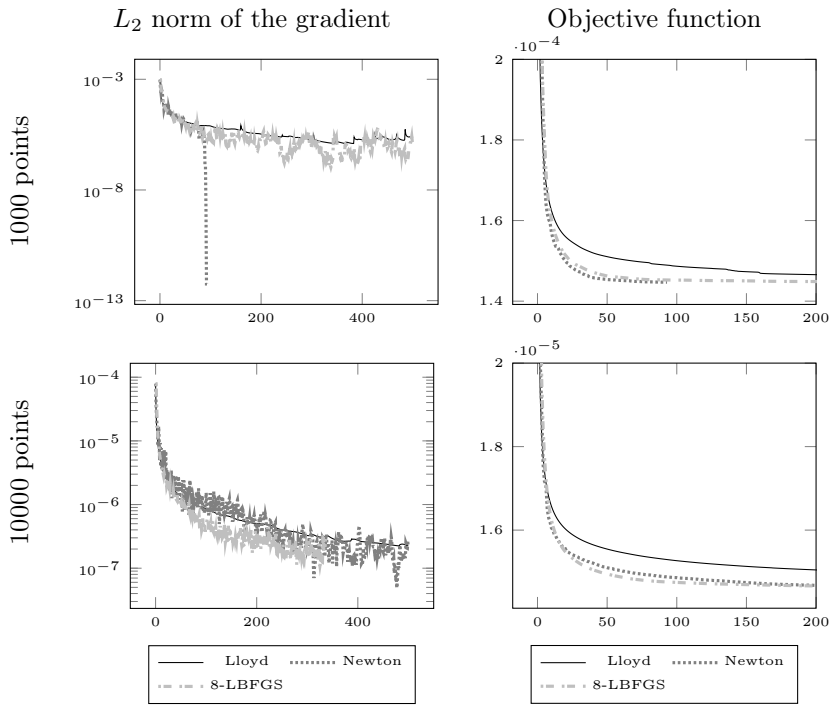

Our interpretation of Figure~\ref{Fig:Voronoi-optim} is the following: in the 1K points problem, the best methods for finding critical points and minimum are, by decreasing order, Newton, $8$-BFGS and Lloyd, which is coherent with theory. For the 10K points problems, the three different methods seem equivalent. Our interpretation of the 10K points behavior is the combination of two factors. First we believe that an augmentation of the number of points reduces the basin of attraction of local minimum. Indeed, in our test the Newton method failed to attain locally convex points (the Hessian always had a negative eigenvalue throughout iterations). The second effect of the augmentation of the number of points is that numerical errors trickle down the algorithm, eventually preventing the Newton method to accurately find the minimum.

\pgfplotstableread[col sep=comma]{OT-1K-normgradient.dat}{\OTuKnormGradient}
\pgfplotstableread[col sep=comma]{OT-1K-costfunction.dat}{\OTuCostFunction}
\pgfplotstableread[col sep=comma]{OT-10K-normgradient.dat}{\OTdKnormGradient}
\pgfplotstableread[col sep=comma]{OT-10K-costfunction.dat}{\OTdCostFunction}

\begin{figure}
\begin{tabular}{ccc}
&$L_2$ norm of the gradient & Objective function \\
\begin{sideways}
\hspace{1cm} $1000$ points 
\end{sideways}
&
\begin{tikzpicture}[every node/.append style={font=\tiny}]
\begin{axis}[ymode=log,width=0.45\textwidth]
\addplot [color=black]
          table[x expr=\thisrow{iteration},y={Lloyd}]{\OTuKnormGradient};
\addplot [color=gray,dash pattern=on 1pt off 1pt,very thick]
          table[x expr=\thisrow{iteration},y={Newton}]{\OTuKnormGradient};
\addplot [color=lightgray,dash pattern=on 2pt off 1pt,very thick]
          table[x expr=\thisrow{iteration},y={LBFGS}]{\OTuKnormGradient};
\end{axis}
\end{tikzpicture}
&
\begin{tikzpicture}[every node/.append style={font=\tiny}]
\begin{axis}[width=0.45\textwidth,ymax=0.0002,xmax=200]
\addplot [color=black]
          table[x expr=\thisrow{iteration},y={Lloyd}]{\OTuCostFunction};
\addplot [color=gray,dash pattern=on 1pt off 1pt,very thick]
          table[x expr=\thisrow{iteration},y={Newton}]{\OTuCostFunction};
\addplot [color=lightgray,dash pattern=on 2pt off 1pt,very thick]
          table[x expr=\thisrow{iteration},y={LBFGS}]{\OTuCostFunction};
\end{axis}
\end{tikzpicture}
\\
\begin{sideways}
\hspace{2cm} $10000$ points 
\end{sideways}
&
\begin{tikzpicture}[every node/.append style={font=\tiny}]
\begin{axis}[ymode=log,width=0.45\textwidth,legend style={ at={(0.5,-0.15)}, anchor=north, legend columns=2}]
\addplot [color=black]
          table[x expr=\thisrow{iteration},y={Lloyd}]{\OTdKnormGradient};
\addplot [color=gray,dash pattern=on 1pt off 1pt,very thick]
          table[x expr=\thisrow{iteration},y={Newton}]{\OTdKnormGradient};
\addplot [color=lightgray,dash pattern=on 3pt off 1pt,very thick]
          table[x expr=\thisrow{iteration},y={LBFGS}]{\OTdKnormGradient};
\legend{Lloyd,Newton,$8$-LBFGS}
\end{axis}
\end{tikzpicture}
&
\begin{tikzpicture}[every node/.append style={font=\tiny}]
\begin{axis}[width=0.45\textwidth,legend style={ at={(0.5,-0.15)}, anchor=north, legend columns=2},ymax=0.00005,xmax=200]
\addplot [color=black]
          table[x expr=\thisrow{iteration},y={Lloyd}]{\OTdCostFunction};
\addplot [color=gray,dash pattern=on 1pt off 1pt,very thick]
          table[x expr=\thisrow{iteration},y={Newton}]{\OTdCostFunction};
\addplot [color=lightgray,dash pattern=on 3pt off 1pt,very thick]
          table[x expr=\thisrow{iteration},y={LBFGS}]{\OTdCostFunction};
\legend{Lloyd,Newton,$8$-LBFGS}
\end{axis}
\end{tikzpicture}
\end{tabular}
\caption{\textit{Blue Noise} problem with 1Kpts (Top) and 10 Kpts (Bottom). Left : norm of the gradient, Right : evolution of the cost function.}
\label{fig::OT-optim}
\end{figure}
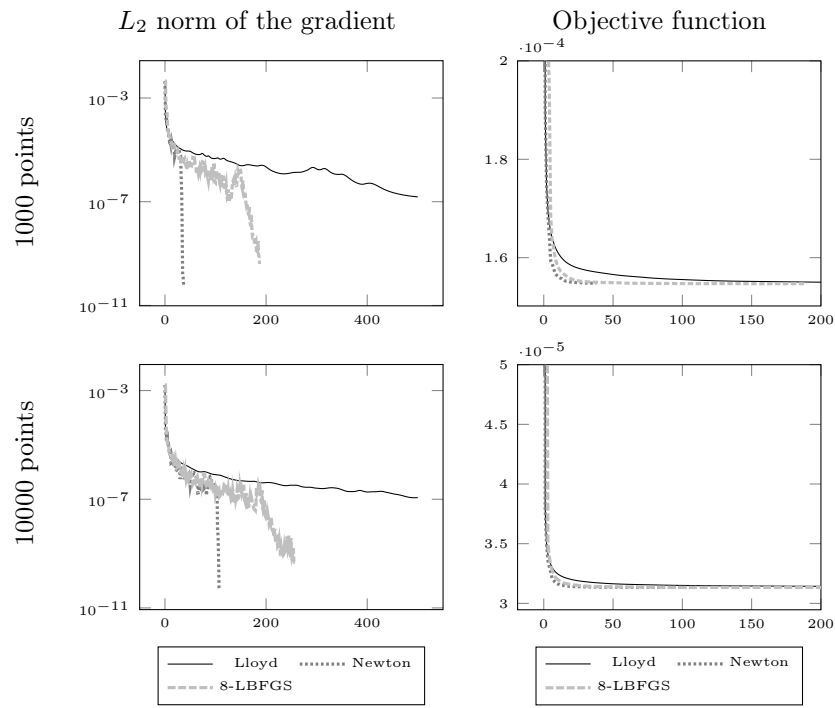

As a conclusion, we find that using second order derivative information in computing Centroidal Voronoi Tessellation (\textit{Stippling} problem) is useful for a small number of points, which renders the application range of this method quite limited. Note that very similar tests have already been performed in \cite{liu2009centroidal}.
The main conclusions of the tests in Figure~\ref{Fig:Voronoi-optim} and Figure~\ref{fig::OT-optim} is that the basin of attraction of the Newton method for the \textit{Blue Noise} problem seems bigger than the one of the \textit{Stippling} problem. Hence a second order method for the \textit{Blue Noise} problem is of interest as the number of points rises.
\subsection{Adding a point}

In order to exhibit the helpfulness of second order method for \textit{Stippling}, we build an example where the classic gradient algorithm fails to converge. Empirically the main drawback of Lloyd algorithm is its lack of \textit{globalisation}. Suppose one has optimized the position of $n$ Dirac masses for the \textit{Blue Noise} or \textit{Stippling} problem and that one adds one mass at some random location and wants to optimize the position of the $n+1$ Dirac masses. Lloyd's algorithm for the \textit{Stippling} problem will converge slowly due to the fact that the new point will modify the Voronoi cells of its neighbours only, whereas the \textit{Blue Noise} functional is global and every Laguerre cell will be modified at the first iteration. Hence  Lloyd's algorithm for the \textit{Stippling} problem has to wait for the information to propagate through each Voronoi cell, like the peeling of an onion, one layer at each iteration. The advantages of the second order method can then be seen, since the Hessian encodes the connectivity and propagates instantly the information. This effect should be less important for the  \textit{Blue Noise} case where information is propagated instantly. In Figure~\ref{fig::addpoint}, we exhibit this effect for the \textit{Blue Noise} and \textit{Stippling} problem. We optimize with 1K pts with a second order method and then test either Lloyd's or Newton's method.

\pgfplotstableread[col sep=comma]{addpoint-normgradient-Stippling.dat}{\AddpointGradientStippling}
\pgfplotstableread[col sep=comma]{addpoint-costfunction-Stippling.dat}{\AddpointCostStippling}
\pgfplotstableread[col sep=comma]{addpoint-normgradient-OT.dat}{\AddpointGradientOT}
\pgfplotstableread[col sep=comma]{addpoint-costfunction-OT.dat}{\AddpointCostOT}

\begin{figure}
\begin{tabular}{cc}
\begin{tikzpicture}[every node/.append style={font=\tiny}]
\begin{axis}[ymode=log,width=0.45\textwidth,xmax=30]
\addplot [color=black]
          table[x expr=\thisrow{iteration},y={Lloyd}]{\AddpointGradientStippling};
\addplot [color=gray,dash pattern=on 1pt off 1pt,very thick]
          table[x expr=\thisrow{iteration},y={Newton}]{\AddpointGradientStippling};
\end{axis}
\end{tikzpicture}
&
\begin{tikzpicture}[every node/.append style={font=\tiny}]
\begin{axis}[width=0.45\textwidth,xmax=30]
\addplot [color=black]
          table[x expr=\thisrow{iteration},y={Lloyd}]{\AddpointCostStippling};
\addplot [color=gray,dash pattern=on 1pt off 1pt,very thick]
          table[x expr=\thisrow{iteration},y={Newton}]{\AddpointCostStippling};

\end{axis}
\end{tikzpicture} \\
\begin{tikzpicture}[every node/.append style={font=\tiny}]
\begin{axis}[ymode=log,width=0.45\textwidth,legend style={ at={(0.5,-0.15)}, anchor=north, legend columns=2},xmax=30]
\addplot [color=black]
          table[x expr=\thisrow{iteration},y={Lloyd}]{\AddpointGradientOT};
\addplot [color=gray,dash pattern=on 1pt off 1pt,very thick]
          table[x expr=\thisrow{iteration},y={Newton}]{\AddpointGradientOT};
\legend{Lloyd,Newton}
\end{axis}
\end{tikzpicture}
&
\begin{tikzpicture}[every node/.append style={font=\tiny}]
\begin{axis}[width=0.45\textwidth,legend style={ at={(0.5,-0.15)}, anchor=north, legend columns=2},xmax=30]
\addplot [color=black]
          table[x expr=\thisrow{iteration},y={Lloyd}]{\AddpointCostOT};
\addplot [color=gray,dash pattern=on 1pt off 1pt,very thick]
          table[x expr=\thisrow{iteration},y={Newton}]{\AddpointCostOT};
\legend{Lloyd,Newton}
\end{axis}
\end{tikzpicture}
\end{tabular}
\caption{Evolution of the cost function (Right) and the norm of the gradient (Left) for the \textit{Stippling} problem (Top) and the \textit{Blue Noise} problem (Bottom) when one point is added.}
\label{fig::addpoint}
\end{figure}
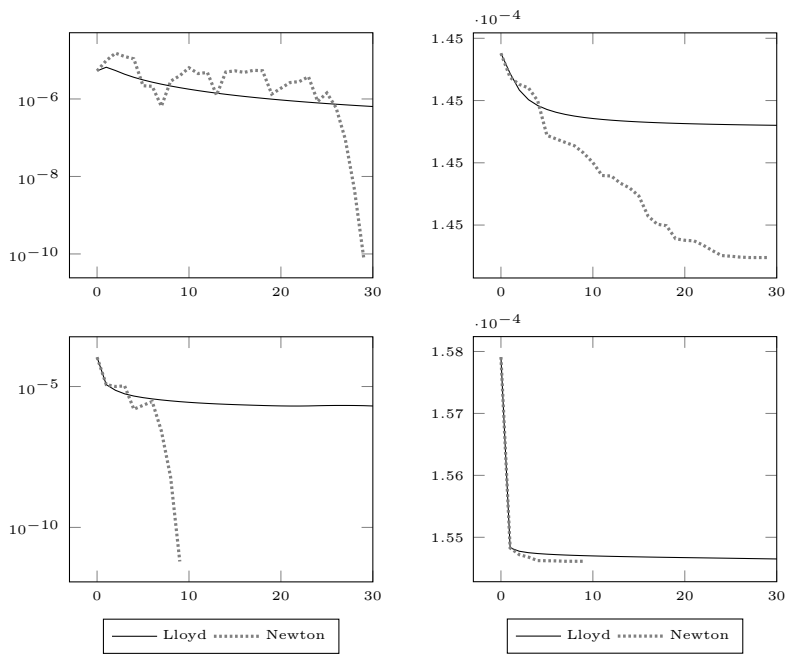

Our interpretation of Figure~\ref{fig::addpoint} lies mainly in the observation of the cost function. Lloyd's method for either the \textit{Blue Noise} or \textit{Stippling} method converge to a critical point in a fraction of the number of iterations needed for the random initialization. The objective function for the \textit{Stippling} problem decreases very slowly for  Lloyd's method compared to the Newton method whereas the decrease of the objective function for the \textit{Blue Noise} is comparable. We interpret this result as the "peeling layers" effect, only seen in the \textit{Stippling} problem, described earlier.
\section{Conclusion}
In this paper we have studied the conditions under which second order differentiability of the semi-discrete optimal transport with respect to position of the Dirac masses  holds for generic cost function $c$. This result encompasses the second order differentiability of the energy of a Voronoi diagram. We have numerically implemented the second order procedure for both the \textit{Blue Noise} and \textit{Stippling} problem. In the \textit{Stippling} problem, the numerical applications are limited by arithmetic precision and small basins of attraction. The \textit{Blue Noise} problem is less sensitive to theses effects. An interpretation of this fact is that the \textit{Blue Noise} problem is global, a change in the position of a mass as an effect on the whole set of masses, whereas in the \textit{Stippling} problem, a mass only sees its direct neighbours. The \textit{Blue Noise} problem is then a more stable problem than the \textit{Stippling} one. 

It is then of the highest interest to understand the smallness of the basins of attraction and the disposition of local minima for the two corresponding problem. It is also the aim of future work to understand optimal transportation between Dirac masses and non-regular background measures (say measures supported by curves) and the corresponding \textit{Blue Noise} problem. Such an application requires to differentiate the semi-discrete optimal transport with respect to parameters that describe the underlying background measure $\nu$.

\bibliography{Biblio}
\bibliographystyle{spmpsci}
\end{document}